\documentclass[11pt,twoside]{amsart}

\usepackage{amssymb,amsmath,amsthm,latexsym}
\usepackage{mathrsfs}
\usepackage{url}
\usepackage{amsfonts}
\usepackage{graphicx}
\usepackage{hyperref}

\usepackage{graphicx}
\usepackage{amssymb,amsthm,amsmath}

\newtheorem{theorem}{Theorem}[section]

\newtheorem{corollary}[theorem] {Corollary}

\newtheorem{lemma} [theorem]{Lemma}

\newtheorem{proposition}[theorem]{Proposition}
\newtheorem{remark}[theorem]{Remark}

\usepackage{enumerate}

\vspace{5cm}
  
\begin{document}

\title{Exploring the Relationships Between the Divisors of Friends of $10$} 
\author[Sagar Mandal]{Sagar Mandal}
\address{Department of Mathematics and Statistics, Indian Institute of Technology Kanpur\\ Kalyanpur, Kanpur, Uttar Pradesh 208016, India}
\email{sagarmandal31415@gmail.com}

\maketitle
\let\thefootnote\relax
\let\thefootnote\relax
\maketitle

\begin{abstract}
A solitary number is a positive integer that shares its abundancy index only with itself. $10$ is the smallest positive integer suspected to be solitary, but no proof has been established so far. In this paper, we prove that not all half of the exponents of the prime divisors of a friend of 10 are congruent to $1$ modulo $3$. Furthermore, we prove that if $F=5^{2a}\cdot Q^2$ ($Q$ is an odd positive integer coprime to $15$) is a friend of $10$, then
$\sigma(5^{2a})+\sigma(Q^2)$ is congruent to $6$ modulo $8$ if and only if $a$ is even, and $\sigma(5^{2a})+\sigma(Q^2)$ is congruent to $2$ modulo $8$  if and only if $a$ is odd. In addition, if we set $Q={\displaystyle \prod_{i=2}^{\omega(F)}}p_{i}^{a_i}$ and $a=a_1$, where $p_i$ are prime numbers, then we establish that $$F>\frac{25}{81}\cdot\prod_{i=1}^{\omega(F)}(2a_i + 1)^2,$$
in particular $F> 625\cdot 9^{\omega(F)-3}.$
\end{abstract}

\section{Introduction}
The sum of positive divisors function $\sigma(F)$ is defined by
$$\sigma(F)=\sigma_{1}(F)=\sum_{\substack{d \mid F\\d>0}}{d}.$$
The abundancy index of $F$ is defined by
$$I(F)=\frac{\sigma(F)}{F}.$$
Note that, $\sigma(F)$, and $I(F)$ are both multiplicative. 
Let $F$ and $E$ be two distinct positive integers such that $I(E)=I(F)$ then we say $E$ and $F$ are friends of each other. A number that has no friends is called a solitary number. Some numbers, such as $10, 14, 15, \text{and}~20$, are believed to be solitary, however no proofs confirming their solitary nature are available to date. $10$ is the smallest positive integer suspected to be solitary. A positive integer $F$ is said to be a friend of $10$ if $I(F)=9/5$ that is $\sigma(F)=\frac{9}{5}F$. In order to find one friend $F$ of 10, many authors gave necessary conditions for the existence of $F$. J. Ward \cite{ward2008does} proved that, $F$ is an odd square positive integer with $\omega(F)\geq 6$ and $5$ is the smallest positive divisor of $F$. In \cite{SS2024}, the authors improved J.Ward's result by proving that in order to become a friend of $10$, $F$ must have $\omega(F)\geq 7$, additionally they established some necessary properties of $F$. In \cite{SS} necessary upper bounds for the second, third, and fourth smallest prime divisors of $F$ have been proposed, further, the author in \cite{SaSo} generalized the paper \cite{SS} and established general upper bounds for all prime divisors of a friend of $10$. In this paper, we establish some relations between the divisors of friends of $10$. Note that, any friend of $10$ can be written as $F = 5^{2a}\cdot Q^2$ (where $Q$ is an odd positive integer coprime to $15$, and $a\geq 1$) then 
$$I(F)=\frac{\sigma(F)}{F}=\frac{\sigma(5^{2a}\cdot Q^2)}{5^{2a}\cdot Q^2}=\frac{\sigma(5^{2a})\cdot \sigma(Q^2)}{5^{2a}\cdot Q^2}=\frac{9}{5}$$
that is
\begin{align}\label{@}
\sigma(5^{2a})\cdot\sigma(Q^2)=9\cdot{5^{2a-1}}\cdot{Q^2}.   \end{align}

Further, if we set $a=a_1$, $p_1=5$, and $Q={\displaystyle\prod_{i=2}^{\omega(F)}p_i^{a_i}}$ ($p_i>5$ for all $2\leq i\leq \omega(F)$ are prime numbers) then from (\ref{@}) we have
$$\sigma(5^{2a_1})\cdot\prod_{i=2}^{\omega(F)}\sigma(p_i^{2a_i})=9\cdot 5^{2a_1-1}\cdot\prod_{i=2}^{\omega(F)}p_i^{2a_i}.$$
A well-known open problem relating to the abundancy index is the existence of odd perfect numbers (an odd positive integer having abundancy index $2$ is called an odd perfect number). Inspired by the papers \cite{odd1970, odd2012} on odd perfect numbers where the authors essentially discussed the exponents of prime divisors of an odd perfect number, we give a similar type of result for a friend of $10$,

\begin{theorem}\label{thm1.1}
If $F=5^{2a_1}\cdot {\displaystyle\prod_{i=2}^{\omega(F)}p_i^{2a_i}}$($p_1=5$) is a friend of 10, then not all $a_i$ are congruent to $13$ modulo $27$.
\end{theorem}
Motivating by the above theorem we prove the following theorem which is even better as we can observe that, if $a_i\equiv 13 \pmod{27}$ holds for all $1\leq i\leq \omega(F)$ then it would imply that $a_i\equiv 13 \pmod{3}$ i.e., $a_i\equiv 1 \pmod{3}$ holds for all $1\leq i\leq \omega(F)$, thus Theorem \ref{thm1.1} follows from Theorem \ref{thm1.2} (recall contrapositive statements). 
\begin{theorem}\label{thm1.2}
If $F=5^{2a_1}\cdot {\displaystyle\prod_{i=2}^{\omega(F)}p_i^{2a_i}}$($p_1=5$) is a friend of 10, then not all $a_i$ are congruent to $1$ modulo $3$.
\end{theorem}
As a consequence of Theorem \ref{thm1.2} we have an alternative proof of the fact that not all $a_i=1$, which can also be deduced from Theorem 1.9\cite{SS2024} (the authors \cite{SS2024} proved that if $F=5^{2a}\cdot Q^2$ is a friend of $10$ then $Q$ must be non squarefree and thus the fact follows).\\
In number theory, congruence relations play a fundamental role as they provide a structured and systematic way to analyze mathematical problems. The following two theorems establish relationships between the divisors of $F$ through specific congruence relations,
\begin{theorem}\label{thm1.3}
If $F=5^{2a}\cdot Q^2$ is a friend of 10, then $\sigma(5^{2a})+\sigma(Q^2)\equiv 6 \pmod 8$ if and only if $a$ is even. 
\end{theorem}
\begin{theorem}\label{thm1.4}
If $F=5^{2a}\cdot Q^2$ is a friend of 10, then $\sigma(5^{2a})+\sigma(Q^2)\equiv 2 \pmod 8$ if and only if $a$ is odd. 
\end{theorem}
We may now ask: why are the congruence relations in Theorems \ref{thm1.3} and \ref{thm1.4} important? We provide an example to understand how they play a significant role in finding a friend of 
$10$. If $F=5^2\cdot Q^2$ is a friend of $10$, then from Theorem \ref{thm1.4} we have
$$\sigma(5^2)+\sigma(Q^2)\equiv 2 \pmod{8}$$
which is basically
$$\sigma(Q^2)\equiv 3 \pmod{8}$$
that is $\sigma(Q^2)=8t+3$ for some $t \in \mathbb{Z}_{\geq0}$. Since $F$ is friend of $10$ we have
\begin{align}\label{why}
Q^2=\frac{5\cdot 31 }{9\cdot 25}\cdot \sigma(Q^2)=\frac{31}{45}\cdot \sigma(Q^2),
\end{align}

since $Q^2\in \mathbb{Z}^{+}$, we must have
$$\sigma(Q^2)\equiv 0 \pmod{45}$$
putting $\sigma(Q^2)=8t+3$ we get
$$8t+3\equiv 0 \pmod{45}$$
that is
$$t\equiv 39 \pmod{45},$$
thus $t=45t'+39$ for some $t'\in \mathbb{Z}_{\geq0}$, using (\ref{why}) we obtain 
\begin{align*}
    Q^2=248t'+217
\end{align*}
therefore it follows that the form of $F$ must be $F=25\cdot(248t'+217)=6200t'+5425$ for some $t'\in \mathbb{Z}_{\geq 0}$.\\
The next theorem provides a lower bound for a friend $F$ of 
$10$, depending on the exponents of the prime divisors of 
$F$. 
\begin{theorem}\label{thm1.5}
If $F=5^{2a_1}\cdot {\displaystyle\prod_{i=2}^{\omega(F)}p_i^{2a_i}}$($p_1=5$) is a friend of 10, then $F$ is strictly larger than $\frac{25}{81}\cdot {\displaystyle \prod_{i=1}^{\omega(F)}(2a_{i}+1)^{2}}$.
\end{theorem}
In \cite{SS2024}, the authors gave an upper bound (see Corollary 1.11) for a friend $F=5^{2a}\cdot Q^2$ of 10: if $\Omega(Q)\leq K$ then $F<5\cdot 6^{(2^{\omega(F)}-1)^2}<5\cdot 6^{(2^{K-2a+1}-1)^2}$, where $\Omega(Q)$ is the total number of prime divisors of $Q$. As a corollary of Theorem \ref{thm1.5} we can give a lower bound for $F$ depending on $\omega(F)$ only. 
\begin{corollary}
If $F$ is a friend of $10$ then
$$F> 625\cdot 9^{\omega(F)-3}.$$
\end{corollary}
\begin{proof}
If $F=5^{2a_1}\cdot{\displaystyle\prod_{i=2}^{\omega(F)}p_i^{2a_i}}$($p_1=5$) is a friend of 10, then not all $a_i=1$ ( by Theorem \ref{thm1.2}) i.e., there exists some $a_j$ such that $a_j\geq2$, thus it implies from Theorem \ref{thm1.5} that
$$F>\frac{25}{81}\cdot 25\cdot 9^{\omega(F)-1}=625\cdot 9^{\omega(F)-3}.$$
\end{proof}
Throughout this paper, we use $p_{1},p_{2},\dots,p_{\omega(n)}$ to denote prime numbers and we consider $a,a_1,\dots,a_{\omega(F)}\in\mathbb{Z}^{+}$.
\section{Proof of The Main Results}\label{Main}
Before we start proving the results, it is useful to state some of the elementary properties of the abundancy index.
\begin{lemma}[\cite{rl,paw}]\end{lemma}
\begin{enumerate}[(a)]
    \item $I(F)$ is weakly multiplicative, that is, if $F$ and $G$ are two coprime positive integers then $I(FG)=I(F)I(G)$.
\item\label{p2} If $\alpha,F$ are two positive integers and $\alpha>1$. Then $I(\alpha F)>I(F)$.
\item If $p_1$, $p_2$, $p_3$,...,$p_n$ are $n$ distinct prime numbers and $\alpha_1$,$\alpha_2$,...,$\alpha_n$ are positive integers then
\begin{align*}
    I\biggl (\prod_{i=1}^{n}p_i^{\alpha_i}\biggl)=\prod_{i=1}^{n}\biggl(\sum_{j=0}^{\alpha_i}p_i^{-j}\biggl)=\prod_{i=1}^{n}\frac{p_i^{\alpha_i+1}-1}{p_i^{\alpha_i}(p_i-1)}.
\end{align*}
\item\label{p4} If $p_{1}$,...,$p_{n}$ are distinct prime numbers and if $q_{1}$,...,$q_{n}$ are distinct prime numbers such that  $p_{i}\leq q_{i}$ for all $1\leq i\leq n$ then for  positive integers $l_1,l_2,...,l_n$ we have
\begin{align*}
   I \biggl(\prod_{i=1}^{n}p_i^{l_i}\biggl)\geq I\biggl(\prod_{i=1}^{n}q_i^{l_i}\biggl).
\end{align*}
\item\label{p5}  If $F={\displaystyle \prod_{i=1}^{n}p_i^{\alpha_i}}$, then $I(F)<{\displaystyle\prod_{i=1}^{n}\frac{p_i}{p_i-1}}$.
\end{enumerate}

 \begin{theorem}[\cite{SS2024}, Theorem 1.3]\label{thm 3.1}
 Let $p,q$ be two distinct prime numbers with  $p^{k-1}\mid\mid (q-1)$ where $k$ is some positive integer. Then $p$ divides $\sigma(q^{2a})$ if and only if $2a+1 \equiv\ 0 \pmod f$ where $f$ is the smallest odd positive integer greater than $1$ such that $q^{f}\equiv 1 \pmod {p^{k}}$. 
\end{theorem}

\begin{proposition}\label{prop1} For any positive integer $a$, we have
  \begin{align*}
        \sigma(5^{2a})\equiv  \left\{
\begin{array}{ll}
      1 \pmod 8 & \text{if~} a\equiv 0 \pmod 4\\
      7 \pmod 8 & \text{if~} a\equiv 1 \pmod 4\\
       5 \pmod 8 & \text{if~} a\equiv 2 \pmod 4\\
        3 \pmod 8 & \text{if~} a\equiv 3 \pmod 4\\
\end{array} 
\right. .
         \end{align*}   
\end{proposition}
\begin{proof}
Since $a$ is a positive integer, we can write $a=4q+r$ for some $q,r\in \mathbb{Z}_{\geq 0},0\leq r\leq 3$, therefore $\sigma(5^{2a})=\frac{5^{2a+1}-1}{4}=\frac{5^{8q+2r+1}-1}{4}$. Let
    \begin{align*}
        \sigma(5^{2a})=\frac{5^{8q+2r+1}-1}{4} \equiv c(q,r) \pmod 8 ~~~~~~~\text{where}~~c(q,r)~~\text{depends on }~q,r.
         \end{align*}
Then we can write 
    \begin{align*}
        5^{8q+2r+1} \equiv 4c(q,r)+1 \pmod{32}.
    \end{align*}
    Since $5^8 \equiv 1 \pmod {32}$ we have, 

    \begin{align*}
        5^{2r+1}\equiv 4c(q,r)+1 \pmod{32}.
        \end{align*}
        This shows that $c(q,r)$ is independent of $q$ i.e., $c(q,r)=c(r)$ only depends on $r$. Thus we have \\
\begin{align*}
    c(r)\equiv  \left\{
\begin{array}{ll}
      1 \pmod 8 & \text{if~} r=0\\
      7 \pmod 8 & \text{if~} r=1\\
       5 \pmod 8 & \text{if~} r=2\\
        3 \pmod 8 & \text{if~} r=3\\
\end{array} 
\right. .
\end{align*}
This completes the proof.
\end{proof}

\begin{remark}\label{*}
 If $F=5^{2a_1}\cdot{\displaystyle\prod_{i=2}^{\omega(F)}p_i^{2a_i}}$($p_1=5$) is a friend of $10$ then
 \begin{align*}
 \sigma(5^{2a_1})\cdot\prod_{i=2}^{\omega(F)}\sigma(p_i^{2a_i})=9\cdot 5^{2a_1-1}\cdot\prod_{i=2}^{\omega(F)}p_i^{2a_i},   
\end{align*}
this shows that $9\mid\mid \sigma(5^{2a_1})\cdot{\displaystyle\prod_{i=2}^{\omega(F)}\sigma(p_i^{2a_i})}=\sigma(F)$ as $p_i\geq 7$, for all $2\leq i\leq\omega(F)$.
\end{remark}
We will make significant use of Remark \ref{*} in the forthcoming proofs. 
\subsection{Proof of Theorem \ref{thm1.1}}
Let $F=5^{2a_1}\cdot{\displaystyle\prod_{i=2}^{\omega(F)}p_i^{2a_i}}$($p_1=5$) be a friend of 10 and also assume that $a_i\equiv 13 \pmod {27}$ for all $1\leq i\leq\omega(F)$. Then in particular,  $a_1\equiv 13 \pmod {27}$ and thus $19\mid \sigma(5^{2a_1})$ follows from Theorem \ref{thm 3.1} (set $p=19$ and $q=5$ in Theorem \ref{thm 3.1} and observe that $2a_1+1\equiv 0\pmod{9}$ and $5^9\equiv 1 \pmod{19}$). Since $19\mid \sigma(5^{2a_1})$ we have $19\mid F$. Without loss of generality suppose that $p_2=19$ then we can write $$F=5^{2a_1}\cdot19^{2a_2}\cdot\prod_{i=3}^{\omega(F)}p_i^{2a_i}.$$
Note that
$$19^{27}\equiv 1 \pmod{486},$$
then
$$19^{27\beta}\equiv 1 \pmod{486},~~\text{for}~\beta \in \mathbb{Z}^{+},$$
which is equivalent to
$$\frac{19^{27\beta}-1}{18}\equiv 0 \pmod{27},$$
that is
$$\sigma(19^{27\beta-1})\equiv 0 \pmod{27}.$$
Therefore for all $\beta\in \mathbb{Z}^{+}$ we have 
$$27\mid \sigma(19^{27\beta-1}).$$
Since $a_2\equiv 13 \pmod{27}$, we have $2a_2+1\equiv 0 \pmod{27}$ therefore setting $\beta= \frac{2a_2+1}{27}\in\mathbb{Z}^{+}$, we obtain 
$$27\mid \sigma(19^{2a_2})$$
which immediately implies that
$$27\mid \sigma(5^{2a_1})\cdot \sigma(19^{2a_2})\cdot\prod_{i=3}^{\omega(F)}\sigma(p_i^{2a_i}),$$
but this contradicts Remark \ref{*}. Therefore, our assumption  that $a_i\equiv 13 \pmod{27}$ is true for all $1\leq i\leq \omega(F)$ is wrong and hence all $a_i$ can not be congruent to $13$ modulo $27$. This completes the proof.\qed\\\\
We now prove a better version of Theorem \ref{thm1.1}. The key steps in proving Theorem \ref{thm1.2} rely on Remark \ref{*} and Theorem \ref{thm 3.1}.
\subsection{Proof of Theorem \ref{thm1.2}}
Let $F=5^{2a_1}\cdot{\displaystyle\prod_{i=2}^{\omega(F)}p_i^{2a_i}}$($p_1=5$) be a friend of 10 and also assume that $a_i\equiv 1 \pmod {3}$ for all $1\leq i\leq\omega(F)$. Then $2a_{i}+1\equiv 0 \pmod{3}$ for all $1\leq i\leq\omega(F)$. Note that, if we consider $p,q=p_i$ with $p^{k-1}\mid\mid (p_{i}-1)$ in Theorem \ref{thm 3.1} and if we have $f=3$ i.e., $p_i^{3}\equiv 1 \pmod{p^{k}}$ then we can use the fact $2a_{i}+1\equiv 0 \pmod{3}$ to conclude that $p\mid \sigma(p_{i}^{2a_i})$. Set $p=31,q=p_1=5$ in Theorem \ref{thm 3.1}, then $k=1$ and observe that $5^3\equiv 1 \pmod{31}$ therefore we have $31\mid \sigma(5^{2a_1})$. Since $31\mid \sigma(5^{2a_1})$ we have $31\mid F$ thus without loss of generality suppose that $p_2=31$. We now show that $3\cdot 331 \mid \sigma(31^{2a_2})$. If we set $p=3,q=p_2=31$ in Theorem \ref{thm 3.1}, then $k=2$, since $31^{3}\equiv 1 \pmod{3^2}$ we have $3\mid \sigma(31^{2a_2})$. Again set $p=331,q=p_2=31$ in Theorem \ref{thm 3.1}, then $k=1$, since $31^3\equiv 1 \pmod{331}$ we have $331\mid \sigma(31^{2a_2})$. Since $331\mid \sigma(31^{2a_2})$ we get $331\mid F$ and so without loss of generality set $p_3=331$. We claim that $3\cdot 7\mid \sigma(331^{2a_3})$.
Set $p=3,q=p_3=331$ in Theorem \ref{thm 3.1}, then $k=2$, since $331^3\equiv 1 \pmod{3^2}$ we have $3\mid \sigma(331^{2a_3})$, again set $p=7,q=p_3=331$ in Theorem \ref{thm 3.1}, then $k=1$, since $331^3\equiv 1 \pmod{7}$ we get $7\mid \sigma(331^{2a_3})$ which proves our claim. Since $7\mid \sigma(331^{2a_3})$ we have $7\mid F$, thus without loss of generality suppose that $p_4=7$. Our final claim is that $3\mid \sigma(7^{2a_4})$. Again setting $p=3,q=p_4=7$ in Theorem \ref{thm 3.1}, we get $k=2$, since $7^3\equiv 1 \pmod{3^2}$ we get $3\mid \sigma(7^{2a_4})$. Observe that, $3\mid\sigma(31^{2a_2}),~3\mid \sigma(331^{2a_3}),~\text{and},~3\mid\sigma(7^{2a_4})$ which immediately implies that $27\mid \sigma(5^{2a_1})\cdot{\displaystyle\prod_{i=2}^{\omega(F)}\sigma(p_i^{2a_i})}=\sigma(F)$ which clearly contradicts Remark \ref{*}. Therefore, our assumption  that $a_i\equiv 1 \pmod{3}$ is true for all $1\leq i\leq \omega(F)$ is wrong and hence all $a_i$ can not be congruent to $1$ modulo $3$. This completes the proof.\qed\\\\
We now derive the congruence relations between the divisors of $F$.
\subsection{Proof of Theorem \ref{thm1.3}}
    Let $F=5^{2a}\cdot Q^2$ be a friend of 10. Assume that $a$ is even. Then from Proposition \ref{prop1} either $\sigma(5^{2a})\equiv 1 \pmod 8$ or $\sigma(5^{2a})\equiv 5 \pmod 8$. Suppose that $\sigma(5^{2a})+\sigma(Q^2) \equiv k \pmod 8$, where $k\in\{0,2,4,6\}$. Since $F=5^{2a}\cdot Q^2$ is a friend of 10, we have
\begin{align}\label{0.1}
\sigma(5^{2a})\cdot\sigma(Q^2)=9\cdot5^{2a-1}\cdot Q^2\equiv5 \pmod 8~~\text{as $Q^2$ is odd and}~a\in\mathbb{Z^+}. 
\end{align}
Now we consider the following two cases,\\
\texttt{Case-1}: If $\sigma(5^{2a})\equiv 1 \pmod 8$ we have
$$\sigma(Q^2) \equiv k-1 \pmod 8,$$
using (\ref{0.1}) we get 
$$k-1\equiv5 \pmod 8$$
that is
$$k\equiv6 \pmod 8.$$
It follows that, for $\sigma(5^{2a})\equiv 1 \pmod 8$, $\sigma(5^{2a})+\sigma(Q^2) \equiv 6 \pmod 8$.\\
\texttt{Case-2}: If $\sigma(5^{2a})\equiv 5 \pmod 8$ then from (\ref{0.1}) we have
$$5\sigma(Q^2)\equiv5 \pmod8.$$
Since $\sigma(Q^2)\equiv k-5 \pmod 8$ we get
$$5(k-5)\equiv5 \pmod 8$$
that is
$$k\equiv6 \pmod 8.$$
This shows that, for $\sigma(5^{2a})\equiv 5 \pmod 8$, $\sigma(5^{2a})+\sigma(Q^2) \equiv 6 \pmod 8$.\\\\
Conversely, let $\sigma(5^{2a})+\sigma(Q^2)\equiv 6 \pmod 8$. If possible, assume that $a$ is odd positive integer. Then from Proposition \ref{prop1} either $\sigma(5^{2a})\equiv 3 \pmod 8$ or $\sigma(5^{2a})\equiv 7 \pmod 8$.\\
\texttt{Case-3}: If $\sigma(5^{2a})\equiv 3 \pmod 8$, using (\ref{0.1}) and the fact $\sigma(5^{2a})+\sigma(Q^2)\equiv 6 \pmod 8$, we get
$$9\equiv 5 \pmod 8$$
which is absurd. Therefore, $\sigma(5^{2a})\equiv 3 \pmod 8$ is impossible.\\
\texttt{Case-4}: If $\sigma(5^{2a})\equiv 7 \pmod 8$, again using (\ref{0.1}) and the fact $\sigma(5^{2a})+\sigma(Q^2)\equiv 6 \pmod 8$, we get
$$49\equiv 5 \pmod 8,$$
which is also absurd. Therefore, $\sigma(5^{2a})\equiv 7 \pmod 8$ is impossible.\\
Thus, our assumption on $a$ is wrong and hence $a$ must be an even positive integer. This completes the proof.\qed\\\\
We proceed in a similar manner to prove the next theorem.
\subsection{Proof of Theorem \ref{thm1.4}}
    Let $F=5^{2a}\cdot Q^2$ be a friend of 10. Assume that $a$ is odd. Then from Proposition \ref{prop1} either $\sigma(5^{2a})\equiv 3 \pmod 8$ or $\sigma(5^{2a})\equiv 7 \pmod 8$. Suppose that $\sigma(5^{2a})+\sigma(Q^2) \equiv k \pmod 8$, where $k\in\{0,2,4,6\}$. Since $F=5^{2a}\cdot Q^2$ is a friend of 10, we have
\begin{align}\label{0.2}
\sigma(5^{2a})\cdot\sigma(Q^2)=9\cdot5^{2a-1}\cdot Q^2\equiv5 \pmod 8~~\text{as $Q^2$~is odd and}~a\in\mathbb{Z^+}.   
\end{align}
Again we consider two cases,\\
\texttt{Case-1}: For $\sigma(5^{2a})\equiv 3 \pmod 8$ we have
$$\sigma(Q^2) \equiv k-3 \pmod 8$$
using (\ref{0.2}) we can write 
$$3(k-3)\equiv5 \pmod 8$$
that is
$$3k\equiv 6 \pmod 8.$$
Since $(3,8)=1$ we get $k\equiv 2 \pmod 8$ which proves that, for $\sigma(5^{2a})\equiv 3 \pmod 8$, $\sigma(5^{2a})+\sigma(Q^2) \equiv 2 \pmod 8$.\\
\texttt{Case-2}: Now for $\sigma(5^{2a})\equiv 7 \pmod 8$ and from (\ref{0.2}) we have
$$7\sigma(Q^2)\equiv5 \pmod8$$
Since $\sigma(Q^2)\equiv k-7 \pmod 8$ we obtain
$$7(k-7)\equiv5 \pmod 8$$
that is
$$k\equiv2 \pmod 8.$$
Therefore, it follows that for $\sigma(5^{2a})\equiv 7 \pmod 8$ we have $\sigma(5^{2a})+\sigma(Q^2) \equiv 2 \pmod 8$.\\\\
Conversely, let $\sigma(5^{2a})+\sigma(Q^2)\equiv 2 \pmod 8$. If possible, assume that $a$ is even positive integer. Then from Proposition \ref{prop1} either $\sigma(5^{2a})\equiv 1 \pmod 8$ or $\sigma(5^{2a})\equiv 5 \pmod 8$.\\
\texttt{Case-3}: If $\sigma(5^{2a})\equiv 1 \pmod 8$, using (\ref{0.2}) and the fact $\sigma(5^{2a})+\sigma(Q^2)\equiv 2 \pmod 8$, we get
$$1\equiv 5 \pmod 8,$$
which is absurd. Therefore, $\sigma(5^{2a})\equiv 1 \pmod 8$ is impossible.\\
\texttt{Case-4}: If $\sigma(5^{2a})\equiv 5 \pmod 8$, again using (\ref{0.2}) and the fact $\sigma(5^{2a})+\sigma(Q^2)\equiv 2 \pmod 8$, we get
$$25\equiv 5 \pmod 8,$$
which is absurd too. Therefore, $\sigma(5^{2a})\equiv 5 \pmod 8$ is impossible.\\
Thus our assumption on $a$ is wrong and hence $a$ must be an odd positive integer. This completes the proof\qed
\subsection{Proof of Theorem \ref{thm1.5}}
Let $F=5^{2a_1}\cdot{\displaystyle\prod_{i=2}^{\omega(F)}p_i^{2a_i}}$($p_1=5$) be a friend of 10. Applying AM-GM inequality, we obtain
\begin{align*}
    \sum_{j=0}^{2a_i}p_i^j>(2a_i+1)\cdot\left( \prod_{j=0}^{2a_{i}}p_i^j\right)^{\frac{1}{2a_{i}+1}}&=(2a_i+1)\cdot p_i^{\frac{(2a_i)(2a_i+1)}{2(2a_i+1)}}\\
    &=(2a_{i}+1)\cdot p_{i}^{a_{i}},
\end{align*}
that is $\sigma(p_i^{2a_i})>(2a_i+1)\cdot p_i^{a_i}$, therefore we can write
\begin{align}\label{2}
\sigma(5^{2a_1})\cdot\prod_{i=2}^{\omega(F)}\sigma(p_i^{2a_i})>\prod_{i=1}^{\omega(F)}(2a_i+1)\cdot 5^{a_1}\cdot\prod_{i=2}^{\omega(F)}p_i^{a_i}.
\end{align}
Since $F$ is a friend of $10$, we have
\begin{align}\label{3}
    \sigma(5^{2a_1})\cdot\prod_{i=2}^{\omega(F)}\sigma(p_{i}^{2a_i})=9\cdot 5^{2a_1-1}\cdot \prod_{i=2}^{\omega(F)}p_{i}^{2a_i},
\end{align}
using (\ref{2}) and (\ref{3}) we get
$$
9\cdot 5^{2a_1-1}\cdot\prod_{i=2}^{\omega(F)}p_{i}^{2a_i}>\prod_{i=1}^{\omega(F)}(2a_i+1)\cdot 5^{a_1}\cdot\prod_{i=2}^{\omega(F)}p_i^{a_i},
$$
multiplying both sides by $5$, we get
$$
9\cdot 5^{2a_1}\prod_{i=2}^{\omega(F)}p_{i}^{2a_i}>5\cdot\prod_{i=1}^{\omega(F)}(2a_i+1)\cdot 5^{a_1}\cdot\prod_{i=2}^{\omega(F)}p_i^{a_i}
$$
that is
$$ 9\cdot F>5\cdot\prod_{i=1}^{\omega(F)}(2a_i+1)\cdot F^{\frac{1}{2}},
$$
squaring both sides and then dividing both sides by $81F$,  we obtain
$$F>\frac{25}{81}\cdot\prod_{i=1}^{\omega(F)}(2a_i+1)^2.$$
This finishes the proof.
\qed
\section{Concluding Remarks}
We have established new results concerning the divisors of friends of $10$. Several positive integers, including 
$14,15,20,$ and many others, are suspected to be solitary, however, no \textbf{general method} is currently known for determining whether a number is friendly or solitary. If any of the suspected solitary numbers up to 
$372$ is actually a friendly number, then its smallest friend must be strictly greater than $10^{30}$\cite{OEIS}.
\section{Acknowledgment}
The author wishes to thank the referee for his/her valuable suggestions towards the improvement of the manuscript.

\end{document}